\documentclass{amsart}
\usepackage[cp1251]{inputenc}
\usepackage[english,russian]{babel}
\usepackage{amsmath}
\usepackage{amssymb}
\usepackage{amsfonts}
\def\udcs{517.518.5} %Здесь автор определяет УДК своей работы

\newtheorem{lemma}{Лемма}
\newtheorem{theorem}{Теорема}
\newtheorem{definition}{Определение}

\newtheorem{proposition}{Предложение}

\begin{document}
УДК \udcs
\thispagestyle{empty}

\title[О равномерных оценках осцилляторных интегралов с гладкой фазой]
{О равномерных оценках осцилляторных интегралов с гладкой фазой}
% Указываем название статьи.
%Сокращенное название указывается в квадратных скобках для колонтитулов,
%если полное название не помещается в строку

\author{ Исроил Акрамович Икромов, Акбар Рахманович Сафаров}
%Указываем авторов
\address{Исроил Акрамович Икромов, %Имя, Отчество, Фамилия первого автора
\newline\hphantom{iii} Акбар Рахманович Сафаров,
% Имя, Отчество, Фамилия второго автора
\newline\hphantom{iii} Институт математики имени В.И.Романовского,% Место работы
\newline\hphantom{iii} ул.Университетский бульвар, 15, % Адрес (улица, дом, строение и т.п.)
\newline\hphantom{iii} 140104, г. Самарканд, Узбекистан}%  Адрес (почтовый индекс, город, страна)
\email{safarov-akbar@mail.ru}% Ваш электронный адрес для переписки

\thanks{\sc Ikromov I.A., Safarov A.R. %  Ф.И.О. авторов на английском языке
Uniform estimates oscillatory integrals with smooth phase}% название статьи на английском языке
\thanks{\copyright \ 2021 Икромов И.А., Сафаров А.Р. }
%\thanks{\rm Работа поддержана РФФИ (грант 03-01-11111)}
\thanks{\it Поступила 1 января 2009 г.}
% (указываем дату отправки, строка будет при получении изменена)

%%%%%%%%%%%%%%%%%%%%%%%%%%%%%%%%%%%%%%%%%%%%%%%%%%%%%%%%%%%%%%%%
\maketitle
{
\small
\begin{quote}
\noindent{\bf Аннотация. } Мы рассмотрим задачу о равномерных оценках осцилляторных интегралов с гладкой фазовой функцией имеющей особенность типа $D_{\infty}$. Оценка являтется точной и являтется аналогом оценок результата В.Н.Карпушкина.
\medskip

\noindent{\bf Abstract. } We consider the problem on uniform estimates for an oscillatory integrals with the smooth phase functions having singularities $D_{\infty} $. The estimate is sharp and analogy to estimates of the work of V.N.Karpushkin.
\medskip

\noindent{\bf Ключевые слова:}{ фаза, деформация, особенность.}
\medskip

\end{quote}
}

\section{Введение}

\begin{definition}
\textit{Осцилляторным интегралом} с гладкой вещественно-значной фазой $f$ и амплитудой $a$ называется интеграл вида:
\[J(\lambda ,f,a)=\int _{{\bf {\mathbb R}}^{n} }a(x)e^{i\lambda f(x)} dx \]
где $a\in C_{0}^{\infty } ({\bf {\mathbb R}}^{n} )$ и $\lambda \in {\bf {\mathbb R}}$.
\end{definition}

 Пусть $U\subset V\subset \mathbb R^{2} -$ограниченные окрестности начала координат, $\overline{U}(\overline{V})$ замыкание $U(V)$. Допустим, что функция $f:\overline{V}\to \mathbb R$ (где $f\in C^{N} \left(\overline{V}\right),$ $\left(N\ge 8\right)$) имеет следующий вид:
\begin{equation} \label{GrindEQ__1_}
f\left(x_{1} ,x_{2} \right)=x_{1} x_{2}^{2} +g\left(x_{1} ,x_{2} \right),
\end{equation}
где $g\in C^{N} \left(V\right),$ такая, что $D^{\alpha } g(0,0)=0$ для всех $\alpha _{1} +\alpha _{2} \le 3,$ здесь $D^{\alpha } $ означает $D^{\alpha } =\frac{\partial ^{\left|\alpha \right|} }{\partial x_{1}^{\alpha _{1} } \partial x_{2}^{\alpha _{2} } } $, $\alpha :=(\alpha _{1} ,\alpha _{2} )\in Z_{+}^{2}-$мультииндекс, $Z_{+}=\{0\}\cup\mathbb{N}$ неотрицательные целые числа.

\begin{definition} \textit{Пусть $F\in C^{N} \left(\bar{V}\right)$ функция такая, что $\left\| F\right\| _{C^{N} \left(\bar{V}\right)} <\varepsilon $, где $N$ некоторое натуральное число и $\varepsilon $ достаточно малое положительное число, а также }$\left\| F\right\| _{C^{N} \left(\overline{V}\right)} =\mathop{\max }\limits_{\overline{V}} \sum\limits_{\left|\alpha \right|\le N}\left|\frac{\partial ^{\left|\alpha \right|} F\left(x_{1} ,x_{2} \right)}{\partial x_{1}^{\alpha _{1} } \partial x_{2}^{\alpha _{2} } } \right| $\textit{. Функция $f+F$ называется деформацией функции $f$} (см.\cite{Karpushkin1983}).
\end{definition}
 
Пусть $\vartheta _{C^{8} \left(\overline{V}\right)} (\varepsilon ):=\{F\in C^{8}\left(\overline{V}\right),\left\| F\right\| _{C^{8} \left(\overline{V}\right)} <\varepsilon\} $ Основным результатом работы является следующая
\begin{theorem}\label{theorem1} Пусть $f\in C^{8} (\overline{V})$ имеет вид \eqref{GrindEQ__1_}. Тогда найдутся положительные числа $\varepsilon,C$ и окрестность $U\subset V$ начала координат, такие, что для произвольных функций $a\in C_{0}^{1} \left(U\right)$ и $F\in\vartheta _{C^{8} \left(\overline{V}\right)} (\varepsilon )$   справедлива следующая оценка:
\[\left|\int _{U}e^{i\lambda \left(f+F\right)} a\left(x\right)dx \right|\le \frac{C\left\| a\right\| _{C^{1} } }{\left|\lambda \right|^{\frac{1}{2} } } .\]
\end{theorem}
1) Теорема является аналогом более общей теоремы В.Н.Карпушкина \cite{Karpushkin1983} (а также см.\cite{Varch})  для достаточно гладких функций.

2) Если $g\equiv0,$ то оценка, полученная в теореме, неулучшаема.

3) Для некоторых $g$ функция $f$ может иметь особенность типа $D_{k}.$ В этом случае из результатов Дюстермаата \cite{Duistermat} можно вывести более точную оценку.

4) Инвариантные оценки с полиномиальной фазой рассмотрены в работе \cite{Safarov}.

\section{Некоторые вспомогательные утверждения}

Сначала мы приведем несколько простых вспомогательных определений.
\begin{definition}\cite{Arn}
Рассмотрим арифметическое пространство $\mathbb{R}^{n}$ с фиксированными координатами $x_{1}, x_{2},\dots, x_{n}.$  Голоморфная функция $f:(\mathbb{R}^{n},0)\rightarrow(\mathbb{R},0)$ называется квазиоднородной функцией степени $d$ с показателями $\alpha_{1}, \alpha_{2},$ $..., \alpha_{n},$ если при любом $\lambda>0$ имеем $f(\lambda^{\alpha_{1}}x_{1}, \lambda^{\alpha_{2}}x_{2},\dots, \lambda^{\alpha_{n}}x_{n})=\lambda^{d}f(x_{1}, x_{2},..., x_{n}).$ Показатели  $\alpha_{s}$ называются весами переменных $x_{s}.$
\end{definition}

\begin{definition}
Cтандартным градиентным идеалом для голоморфного функции $f:(\mathbb{R}^{n},0)\rightarrow(\mathbb{R},0)$  называется $I_{\nabla f}=\left(\frac{\partial f}{\partial x_{1}}, \frac{\partial f}{\partial x_{2}},...,\frac{\partial f}{\partial x_{n}}\right).$

\end{definition}
 Пусть  функция $f\left(x_{1} ,x_{2} \right)$ удовлетворяет следующим условиям:
\[1)  \frac{\partial ^{\left|\alpha \right|} f\left(0,0\right)}{\partial x_{1}^{\alpha _{1} } \partial x_{2}^{\alpha _{2} } } =0,  \left|\alpha \right|=\alpha _{1} +\alpha _{2} \le 2.\]
2) допустим корни уравнения $f_{3} \left(x_{1} ,x_{2} \right)=0$ на $S_{1} $ (где  $f_{3} $ отрезок Тейлора функции $f$ порядка 3  и $S_{1} $ единичная окружность в $\mathbb R^{2} $ с центром в начале координат) состоят одного простого и двукратного корня.

 Тогда функция $f$, линейным преобразованием, приводится к виду $f\left(x\left(u\right)\right)=u_{1} u_{2}^{2} +g_{1}\left(u_{1} ,u_{2} \right)$, где $g_{1}\left(u_{1} ,u_{2} \right)$ - некоторая функция удовлетворяющая условию $D^{\alpha } g_{1}\left(0,0\right)=0,$ при всех $\left|\alpha \right|\le 3$, аналогичное утверждение доказано для однородных многочленов третьей степени в работе \cite{Arn} (стр.147). Главная часть относительно веса $\left(\frac{1}{3},\frac{1}{3}\right)$ функции $f$ обозначается через $f_{\pi }$. Таким образом, без ограничение общности, мы можем считать, что $f_{\pi } (x_{1} ,x_{2} )=x_{1} x_{2}^{2} $.

\textbf{ }Следуя \cite{Karpushkin1983}  обозначим через $E_{d} $ линейное пространство полиномов степени меньше $d$ относительно веса $\left(\frac{1}{3} ,\frac{1}{3} \right)$ и   \textit{$I_{\nabla f_{\pi } } $   }стандартный градиентный идеал функции $f_{\pi } $.

\begin{definition}Координатное подпространство $B\subset E_{1} $ называется нижным версальным, если $\left(I_{\nabla f_{\pi } } \bigcap E_{1} \right)\oplus B=E_{1} $ (т.е. $I_{\nabla f_{\pi } } \bigcap E_{1} \bigcap B=0,\, \, \left(I_{\nabla f_{\pi } } \bigcap E_{1} \right)+B=E_{1} $)).
\end{definition}
Легко показать, что $B=\left\langle 1,x_{1} ,x_{2} ,x_{1}^{2} \right\rangle$ является версальным подпространством $B$ для функции $f_{\pi } \left(x_{1} ,x_{2} \right)=x_{1} x_{2}^{2}.$

Пусть $\sum\limits_{\frac{m_{1} }{3} +\frac{m_{2} }{3} <d}s_{m} x^{m}  $ отрезок ряда Тейлора в точке 0 функции $F.$ Положим $\pi _{d} (F)=\sum s_{m}  x^{m} ,$ где $0\leq\frac{m_{1} }{3} +\frac{m_{2} }{3} <d$. Таким образом, $\pi _{1} $ определяет отображение пространства $C^{N} \left(V\right)$ на пространство $E_{d},$ где $d\le \frac{N}{3} .$

Следующее предложение о возможности гладко выбрать замену координат является аналогом теоремы версальности. Аналогом деформации функции $f$ является $f+F$, где $F\in \vartheta _{C^{8} \left(\overline{V}\right)} (\varepsilon )$ (деформация с бесконечным числом параметров). Аналогом версальной деформации $f$ является $f+F$, где $F\in \vartheta _{C^{8} \left(\overline{V}\right)} (\varepsilon )$, $\pi _{1} F\in B.$ Здесь $B-$нижнее версальное подпространство \cite{Arn}.

\begin{lemma}\label{lem1}
Пусть $z\in C^{8}(\overline{V})$ некоторая функция. Функция $(f+F)(y+z(y))$ записывается в виде
$$(f+F)(y+z(y))=f(0)+F(0)+\alpha_{10}(z)y_{1}+\alpha_{01}(z)y_{2}+$$$$+\alpha_{20}(z)y_{1}^{2}+\alpha_{11}(z)y_{1}y_{2}+
\alpha_{02}(z)y_{2}^{2}+\sum\limits_{i_{1}+i_{2}=3}\alpha_{i_{1}i_{2}}(z)y_{1}^{i_{1}}y_{2}^{i_{2}}+y_{1}y_{2}^{2},$$
где $\alpha_{10},\alpha_{01}, \alpha_{20}-$функционалы от $F$ и $\alpha_{i_{1}i_{2}}:C^{3}(U)\rightarrow C^{3}(V)$ операторы причем $2\leq i_{1}+i_{2}=3,$ выполняется $\|\alpha_{i_{1}i_{2}}\|_{C^{2}}\leq C\varepsilon,$ при условии $F\in \vartheta _{C^{8} \left(\overline{V}\right)} (\varepsilon ).$
\end{lemma}
\begin{proof}
Представим функцию  $f+F$ в следующем виде:
\[\begin{array}{l} {f+F=s_{10} x_{1} +s_{01} x_{2} +s_{20} x_{1}^{2} +2s_{11} x_{1} x_{2} +s_{02} x_{2}^{2} +s_{30} (x_{1} ,x_{2} )x_{1}^{3} +} \\ {+s_{21} (x_{1} ,x_{2} )x_{1}^{2} x_{2} +(1+s_{12} (x_{1} ,x_{2} ))x_{1} x_{2}^{2} +s_{03} (x_{1} ,x_{2} )x_{2}^{3} ,} \end{array}\]
где $s_{10} =\frac{\partial F(0,0)}{\partial x_{1} } ,\, \, s_{01} =\frac{\partial F(0,0)}{\partial x_{2} } ,\, \, s_{20} =\frac{1}{2} \frac{\partial ^{2} F(0,0)}{\partial x_{1}^{2} } ,\, \, s_{11} =\frac{1}{2} \frac{\partial ^{2} F(0,0)}{\partial x_{1} \partial x_{2} } ,\, \, s_{02} =\frac{1}{2} \frac{\partial ^{2} F(0,0)}{\partial x_{2}^{2} } ,$  $s_{k_{1}k_{2}}(x_{1},x_{2}):=s_{k_{1}k_{2}}=\int _{0}^{1}(1-u)^{2} \frac{\partial ^{3} (F+g)(ux_{1},ux_{2} )}{\partial x_{1}^{k_{1}} \partial x_{2}^{k_{2}}}du,$ $k_{1}+k_{2}=3.$

Сделаем замену $x_{1} -z_{1}(y_{1},y_{2})=y_{1} $, $x_{2} -z_{2}(y_{1},y_{2})=y_{2} $ и используем следующие разложения
$$z_{1}(y)=z_{1}^{0}+a_{10}y_{1}+a_{01}y_{2}+a_{20}y_{1}^{2}+a_{11}y_{1}y_{2}+a_{02}y_{2}+r(y)$$ и
$$z_{2}(y)=z_{2}^{0}+b_{10}y_{1}+b_{01}y_{2}+b_{20}y_{1}^{2}+b_{11}y_{1}y_{2}+b_{02}y_{2}+r(y),$$
где
$z_{1}^{0}=z_{1}(0,0)$, $a_{10}=\frac{\partial z_{1}(0,z)}{\partial y_{1}},$ $a_{20}=\frac{\partial z_{1}(0,z)}{\partial y_{2}},$
$a_{20}=\frac{\partial^{2} z_{1}(0,z)}{\partial y_{1}^{2}},$ $a_{02}=\frac{\partial^{2} z_{1}(0,z)}{\partial y_{2}^{2}},$
$a_{11}=\frac{\partial^{2} z_{1}(0,z)}{\partial y_{1}\partial y_{2}},$ $z_{2}^{0}=z_{2}(0,0)$, $b_{10}=\frac{\partial z_{2}(0,z)}{\partial y_{1}},$ $b_{20}=\frac{\partial z_{2}(0,z)}{\partial y_{2}},$
$b_{20}=\frac{\partial^{2} z_{2}(0,z)}{\partial y_{1}^{2}},$ $b_{02}=\frac{\partial^{2} z_{2}(0,z)}{\partial y_{2}^{2}},$
$b_{11}=\frac{\partial^{2} z_{2}(0,z)}{\partial y_{1}\partial y_{2}}.$

Разложим функцию $(F+g)(y+z(y))$ по формуле Тейлора в точке $(y_{1},y_{2})=(0,0)$

Тогда получим
$$(f+F)(y+z(y))=\alpha_{00}(z)+\alpha_{10}(z)y_{1}+\alpha_{01}(z)y_{2}+$$$$+\alpha_{20}(z)y_{1}^{2}+\alpha_{11}(z)y_{1}y_{2}+
\alpha_{02}(z)y_{2}^{2}+\sum\limits_{i_{1}+i_{2}=3}\alpha_{i_{1}i_{2}}(z)y_{1}^{i_{1}}y_{2}^{i_{2}}+y_{1}y_{2}^{2},$$
где
$\alpha_{00}(z)=f(0)+F(0),\alpha_{10}(z),\alpha_{01}(z),\alpha_{20}(z)$ некоторые функционалы от $z$ и $\alpha_{02}(z),\alpha_{11}(z):C^{3}(U)\rightarrow C^{3}(V)$ операторы имеющие вид:
$$\alpha_{02}(z)=z_{1}+\tilde{\Phi}_{1}(z,F),\,\,\,\alpha_{11}(z)=2z_{2}+\tilde{\Phi}_{2}(z,F),$$
здесь $\tilde{\Phi}_{j}:C^{3}(U)\rightarrow C^{3}(V)$ некоторые гладкие операторы, удовлетворяющие условиям  $\tilde{\Phi}_{j}(z,0)\equiv0,\,\,(j=1,2)$.
\end{proof}

\begin{proposition}\label{prop1}
 Существует положительное число $\varepsilon >0$ такое, что для\textbf{ }любого\textbf{ } $\left\| F\right\| _{C^{4} \left(\overline{V}\right)} <\varepsilon $ найдется такое отображение $(z_{1} ,z_{2} ):=\left(z_{1} (F),z_{2} (F)\right)\in C^{4} \left(U\to \mathbb R^{2} \right)\, ,$ определенное в некоторой окрестности $U$, для которого   справедливо следующее равенство
\[\pi _{1} \left(f\left(y_{1} +z_{1} ,y_{2} +z_{2} \right)+F\left(y_{1} +z_{1} ,y_{2} +z_{2} \right)\right)=\tilde{c}_{0} (F)+\tilde{c}_{1} (F)y_{1} +\tilde{c}_{2} (F)y_{2} +\tilde{c}_{3} (F)y_{_{1} }^{2} ,\]
где\textbf{ }$\pi _{1} (\cdot )-$\textbf{ }проектирование пространства $C^{4} (V)$ на пространство $E_{1} .$\textbf{ }
\end{proposition}
 Теперь  рассмотрим следующие функциональные уравнения относительно $\left(z_{1} ,z_{2} \right):$
\begin{equation} \label{GrindEQ__2_}
\Phi _{1} (y,F,z):=\alpha_{11}(z)=0,\Phi _{2} (y,F,z):=\alpha_{02}(z)=0.
\end{equation}

Приведем вспомогательную лемму.

\begin{lemma}\label{lem2} Для непрерывных операторов $\Phi _{1} (y,F,z)$, $\Phi _{2} (y,F,z)$ в пространстве $U_{1} \times C^{4} (V_{1} )\times U_{2} $ существует частная  производная по $z$ и они дифференцируемы по Фреше в точке 0, где $U_{1} \subset \mathbb R^{2} ,\, U_{2} \subset \mathbb R^{2} .$
\end{lemma}
\begin{proof} Ради определенности, покажем существование  частных  производных по $z$ и  дифференцируемость оператора $\Phi _{1} (y,F,z)$, для функции  $\Phi _{2} (y,F,z)$ доказательство совершенно аналогично.

Так как функции \textit{$F\in C^{8} \left(\bar{V}\right)$ }и  \textit{$g\in C^{8} \left(\bar{V}\right)$} то отсюда вытекает  дифференцируемость оператора $\Phi _{1} (y,F,z)$.

Существование производных отображения $\Phi _{2} (y,F,z)$  рассматривается аналогично.
\end{proof}
\begin{lemma}\label{lem3} Операторы $\Phi _{1} (y,F,z)$, $\Phi _{2} (y,F,z)$ удовлетворяют следующим условиям: 1)$\Phi _{1}(y,0,0)\equiv0,\, \, \Phi _{2}(y,0,0)\equiv0.$ 2) $\left|\begin{array}{l} {\frac{\partial \Phi _{1} }{\partial z_{1} } \, \, \, \, \, \frac{\partial \Phi _{1} }{\partial z_{2} } } \\ {\frac{\partial \Phi _{2} }{\partial z_{1} } \, \, \, \, \frac{\partial \Phi _{2} }{\partial z_{2} } } \end{array}\right|\ne 0.$
\end{lemma}
\begin{proof} Из явного вида операторов $\Phi_{1},\Phi_{2}$ вытекает выполнения соотношения 1)$\Phi _{1}(y,0,0)\equiv0,\, \, \Phi _{2}(y,0,0)\equiv0.$ Очевидно, что $\left|\frac{\partial \Phi _{1} (0)}{\partial z_{1} } \frac{\partial \Phi _{2} (0)}{\partial z_{2} } -\frac{\partial \Phi _{1} (0)}{\partial z_{2} } \frac{\partial \Phi _{2} (0)}{\partial z_{1} } \right|=2\ne 0$ и следовательно выполнено второе утверждение. Лемма \ref{lem3} доказана.

Перейдем к доказательству предложения. Так как операторы $\Phi _{1} (y,F,z)$ и $\Phi _{2} (y,F,z)$, согласно лемм \ref{lem1} и \ref{lem2}, удовлетворяют условиям теоремы о неявных отображениях, то, согласно  этой теореме, найдется  решение $z_{1} =z_{1} \left(y_{1} ,y_{2} ,F(y)\right)$, $z_{2} =z_{2} \left(y_{1} ,y_{2} ,F(y)\right)$ уравнения \eqref{GrindEQ__2_} и они являются гладкими функциями, в зависимости от гладкости отражения $F$.
\end{proof}
\textbf{Замечание 1.} Отметим, что если $F\in C^{3+k} $ и $f\in C^{3+k} $ то $z(y)\in C^{k} $.

\section{О разбиении единицы}

Осцилляторный интеграл оценивается с помощью разбиения единицы.

  Пусть $k=\left(\frac{k_{1} }{3} ,\frac{k_{2} }{3} \right)$ и $\tau >0$ фиксированное число. Рассмотрим отображение

\noindent $\delta _{\tau }^{k} :\mathbb R^{2} \to \mathbb R^{2} $ определенное формулой:
\[\delta _{\tau} (x)=\left(\tau x_{1} ,\tau x_{2} \right).\]
Введем  функцию $\beta (x)$, удовлетворяющую условиям:\\
1) $\beta \in C^{\infty } \left(\mathbb R^{2} \right),$\\
2) $0\le \beta (x)\le 1$ для любого $x\in \mathbb R^{2} $,\\
3) $\beta (x)=\left\{\begin{array}{l} {1,\, \, \, \, \, \text{при}\, \, |x|\le 1,} \\ {0,\, \, \, \, \text{при}\, \, \left|\delta _{2^{-1} } (x)\right|\ge 1} \end{array}\right. $

Существование такой функции доказано в \cite{Vladim} (а также \cite{Sogge}).

\noindent Пусть
$$\chi (x)=\beta (x)-\beta (\delta _{2} (x)).$$

\noindent Основные свойства функции $\chi (x)$содержатся в следующей лемме.
\begin{lemma}\label{lem4}
\textit{Функция $\chi (x)$ удовлетворяет следующим условиям:}

\begin{enumerate}
\item \textit{ Для произвольного фиксированного $x$ справедливо равенство}
\[\beta (\delta _{2^{-\nu_{0}}} (x))+\sum _{\nu =v_{0} }^{\infty }\chi \left(\delta _{2^{-\nu } } (x)\right)=1 .\]

\item \textit{Для произвольного} \textit{$x\ne 0$ существует $\nu _{0} =\nu _{0} (x)$ такое, что при любом $\nu \notin [\nu _{0} ,\nu _{0} +4]$}
\[\chi \left(\delta _{2^{\nu } } (x)\right)=0.\]

\item  \textit{Для произвольного $\nu _{0} $ существует $\varepsilon >0$ такое, что }$\chi \left(\delta _{2^{\nu } } (x)\right)=0$\textit{ при любом $\nu <\nu _{0} $ и $\left|x\right|\ge \varepsilon $.}
\end{enumerate}
\end{lemma}
\noindent Лемма \ref{lem4} доказана в работе  \cite{Sogge}.

\section{Доказательство основного результата}

Так как функция имеет вид $f\left(x_{1} ,x_{2} \right)=x_{1} x_{2}^{2} +g\left(x_{1} ,x_{2} \right)$, то применяя предложение \ref{prop1} для $f+F$ и без ограничение общности мы представим коэффициенты $s_{ij} $ без «шапочки» и получим
\begin{equation} \label{GrindEQ__3_}
f+F=s_{10} y_{1} +s_{01} y_{2} +s_{20} y_{1}^{2} +y_{1} y_{2}^{2} +s_{30} y_{1}^{3} +s_{21} y_{1}^{2} y_{2} +s_{12} y_{1} y_{2}^{2} +s_{03} y_{2}^{3} +R_{4} \left(y_{1} ,y_{2} \right),
\end{equation}
где $R_{4} \left(y_{1} ,y_{2} \right)$ остаточный член. Теперь оценим интеграл $J.$ Сначала введем "квазирасстояние"  $\rho =\left|s_{10} \right|^{\frac{3}{2} } +\left|s_{01} \right|^{\frac{3}{2} } +\left|s_{20} \right|^{3}$ и  в интеграле \eqref{GrindEQ__1_} с фазовой функцией \eqref{GrindEQ__3_} сделаем замену переменных $y_{1} =\rho ^{\frac{1}{3} } \tau _{1} ,\, y_{2} =\rho ^{\frac{1}{3} } \tau _{2} $. Тогда получим:
\[J\left(\lambda \right)=\rho ^{\frac{2}{3} } \int _{R^{2} }a\left(\rho ^{\frac{1}{3} } \tau _{1} ,\rho ^{\frac{1}{3} } \tau _{2} \right)e^{i\lambda \rho\Phi} d\tau , \]
где $\Phi= \frac{s_{10} }{\rho ^{\frac{2}{3} } } \tau _{1} +\frac{s_{01} }{\rho ^{\frac{2}{3} } } \tau _{2} +\frac{s_{20} }{\rho ^{\frac{1}{3} } } \tau _{1}^{2} +\tau _{1} \tau _{2}^{2} +s_{30} \tau _{1}^{3} +s_{21} \tau _{1}^{2} z_{2} +s_{12} \tau _{1} \tau _{2}^{2} +s_{03} \tau _{2}^{3} +\frac{1}{\rho } R_{4} \left(\rho ^{\frac{1}{3} } \tau _{1} ,\rho ^{\frac{1}{3} } \tau _{2} \right).$
Применим лемму \ref{lem4}, т.е. разбиение единицы, для интеграла $J(\lambda )$ и получим разложение в следующем виде:
\[J\left(\lambda \right)=J_{0} (\lambda )+\sum _{k=k_{0} }^{\infty }J_{k} \left(\lambda \right) ,\]
где
\[J_{k} \left(\lambda \right)=\rho ^{\frac{2}{3} } \int _{\mathbb R^{2} }a\left(\rho ^{\frac{1}{3} } \tau _{1} ,\rho ^{\frac{1}{3} } \tau _{2} \right)\chi \left(2^{-\frac{k}{3} } \tau _{1} ,2^{-\frac{k}{3} } \tau _{2} \right)e^{i\lambda \rho \Phi} d\tau , \]
\[J_{0} \left(\lambda \right)=\rho ^{\frac{2}{3} } \int _{\mathbb R^{2} }a\left(\rho ^{\frac{1}{3} } \tau _{1} ,\rho ^{\frac{1}{3} } \tau _{2} \right)\beta_{0} (\delta _{2^{-k_{0}}} (x))e^{i\lambda \rho \Phi} d\tau . \]

Сначала оценим интеграл $J_{k} \left(\lambda \right)$. В этом интеграле $J_{k} \left(\lambda \right)$ сделаем замену переменных $2^{-\frac{k}{3} } \tau _{1} =t_{1} $, $2^{-\frac{k}{3} } \tau _{2} =t_{2} $ и  получим
\[J_{k} \left(\lambda \right)=2^{\frac{2k}{3} } \rho ^{\frac{2}{3} } \int _{\mathbb R^{2} }a\left(2^{\frac{k}{3} } \rho ^{\frac{1}{3} } t_{1} ,2^{\frac{k}{3} } \rho ^{\frac{1}{3} } t_{2} \right)\chi \left(t_{1} ,t_{2} \right)e^{i\lambda 2^{k} \rho \Phi _{k} \left(t,s,\rho \right)} dt ,\]
где фазовая функция имеет  вид
\[\begin{array}{l} {\Phi _{k} \left(t,s,\rho \right)=2^{-\frac{2k}{3} } \sigma _{10} t_{1} +2^{-\frac{2k}{3} } \sigma _{01} t_{2} +2^{-\frac{k}{3} } \sigma _{20} t_{1}^{2} +t_{1} t_{2}^{2} +} \\ {+s_{30} t_{1}^{3} +s_{21} t_{1}^{2} t_{2} +s_{12} t_{1} t_{2}^{2} +s_{03} t_{2}^{3} +2^{-k} \frac{1}{\rho } R_{4} \left(2^{\, \frac{k}{3} } \rho ^{\frac{1}{3} } t_{1} ,2^{\frac{k}{3} } \rho ^{\frac{1}{3} } t_{2} \right)\, ,} \end{array}\]
здесь $\sigma _{10} =\frac{s_{10} }{\rho ^{\frac{2}{3} } } $, $\sigma _{01} =\frac{s_{01} }{\rho ^{\frac{2}{3} } } $, $\sigma _{20} =\frac{s_{20} }{\rho ^{\frac{1}{3} } } $, $R_{4} \left(2^{\, \frac{k}{3} } \rho ^{\frac{1}{3} } t_{1} ,2^{\frac{k}{3} } \rho ^{\frac{1}{3} } t_{2} \right)=\frac{2^{\, \frac{4k}{3} } \rho ^{\frac{4}{3} } }{6} (s_{40} t_{1}^{4} +s_{31} t_{1}^{3} t_{2} +s_{22} t_{1}^{2} t_{2}^{2} +s_{13} t_{1} t_{2}^{3} +s_{04} t_{2}^{4} )$,

\noindent  где $s_{40} \left(t,s,\rho \right)=\frac{1}{6} \int _{0}^{1}(1-u)^{3} \frac{\partial ^{4} \Phi _{k} \left(ut,s,\rho \right)}{\partial t_{1}^{4} } du,\,  $ $s_{31} (t,s,\rho )=\frac{1}{6} \int _{0}^{1}(1-u)^{3} \frac{\partial ^{4} \Phi _{k} \left(ut,s,\rho \right)}{\partial t_{1}^{3} \partial t_{2}^{} } du,\,  $ $s_{22} \left(t,s,\rho \right)=\frac{1}{6} \int _{0}^{1}(1-u)^{3} \frac{\partial ^{4} \Phi _{k} \left(ut,s,\rho \right)}{\partial t_{1}^{2} \partial t_{2}^{2} } du,\,  $$s_{13} \left(t,s,\rho \right)=\frac{1}{6} \int _{0}^{1}(1-u)^{3} \frac{\partial ^{4} \Phi _{k} \left(ut,s,\rho \right)}{\partial t_{1}^{} \partial t_{2}^{3} } du,\,  $$s_{04} \left(t,s,\rho \right)=\frac{1}{6} \int _{0}^{1}(1-u)^{3} \frac{\partial ^{4} \Phi _{k} \left(ut,s,\rho \right)}{\partial t_{2}^{4} } du.\,  $

        Мы можем считать (в зависимости от носителя амплитуды $\chi _{0} $, по лемме \ref{lem4}), что  число $k_{0} $ достаточно большое.

 Сначала рассмотрим случай, когда нет осцилляции. Пусть $\left|2^{k} \lambda \rho \right|\le L$, где $L$ большое фиксированное число. Тогда из тривиальной оценки интеграла получим:
\begin{equation}\label{F4}
\left|J_{k} \right|\le \frac{2^{\, \frac{2k}{3} } \rho ^{\frac{2}{3} } A}{\left|2^{k} \lambda \rho \right|^{\frac{1}{2} } } =\frac{2^{\, \frac{k}{6} } \rho ^{\frac{1}{6} } A}{\left|\lambda \right|^{\frac{1}{2} } } .
\end{equation}

Пусть теперь $\left|2^{k} \lambda \rho \right|> L$ и $k>k_{0}$ достаточно большое число. Тогда, по условию, $\Phi _{k}$ может быть рассмотрена как
малая деформация функции $\tau_{1}\tau_{2}^{2}$, причем $(\tau_{1},\tau_{2})\in D:=$supp$(\chi)=\{\frac{1}{2}\leq|\tau|\leq2\}.$
Очевидно, что если $\tau^{0}\in D$ фиксированная точка и $\tau_{2}^{0}\neq0,$ то эта точка не является критической. Если $\chi^{0}$ срезающая функция (т.е. функция носитель которой находится в достаточно малой окрестности этой точки), то интеграл
$$J_{k}^{\chi^{0}} \left(\lambda \right):=2^{\frac{2k}{3} } \rho ^{\frac{2}{3} } \int _{R^{2} }a\left(2^{\frac{k}{3} } \rho ^{\frac{1}{3} } t_{1} ,2^{\frac{k}{3} } \rho ^{\frac{1}{3} } t_{2} \right)\chi \left(t_{1} ,t_{2} \right)e^{i\lambda 2^{k} \rho \Phi _{k} \left(t,s,\rho \right)}\chi^{0}(t) dt,$$
тривиально оценивается интегрированием по частям и имеет место неравенство \eqref{F4}.

Если $\tau_{2}^{0}=0,$ то $\tau_{1}^{0}\neq0.$ В этом случае, используя лемму Ван дер Корпута \cite{VanDer} (более общее утверждение содержит в \cite{Carberry}), снова имеем оценку вида \eqref{F4}.

Так как на носителе амплитуды $2^{\frac{k}{3} } \rho ^{\frac{1}{3} } <1,$ то
\[\frac{1}{\left|\lambda \right|^{\frac{1}{2} } } \sum _{\left|2^{k} \lambda \rho \right|\le 1}\left|J_{k} \right| \le \frac{c}{\left|\lambda \right|^{\frac{1}{2} } } \sum _{2^{k}\rho \le 1}2^{\, \frac{k}{6} } \rho ^{\frac{1}{6} }  \le \frac{c}{\left|\lambda \right|^{\frac{1}{2} } } .\]

Теперь рассмотрим оценку интеграла $J_{0} (\lambda )$. Рассмотрим следующие случаи для параметров $\sigma $.

Введем квазисферу $\rho(\sigma):=\{\left|\sigma _{10} \right|^{\frac{2}{3} } +\left|\sigma _{01} \right|^{\frac{2}{3} } +\left|\sigma _{20} \right|^{\frac{1}{3} } =1\}$ и рассмотрим фазовую функцию
\[\begin{array}{l} {\Phi _{0} \left(\tau ,\sigma ,\rho \right)=\sigma _{10} \tau _{1} +\sigma _{01} \tau _{2} +\sigma _{20} \tau _{1}^{2} +\tau _{1} \tau _{2}^{2} +s_{30} \tau _{1}^{3} +s_{21} \tau _{1}^{2} \tau _{2} +s_{12} \tau _{1} \tau _{2}^{2} } \\ {+s_{03} \tau _{2}^{3} +\frac{1}{\rho } \frac{\rho ^{\frac{4}{3} } }{6} (s_{40} \tau _{1}^{4} +s_{31} \tau _{1}^{3} \tau _{2} +s_{22} \tau _{1}^{2} \tau _{2}^{2} +s_{13} \tau _{1} \tau _{2}^{3} +s_{04} \tau _{2}^{4} ).} \end{array}\]

Отметим, что на квазисфере $c_{1}\leq\left|\sigma \right|\le c_{2}$, где $c_{1}, c_{2}-$фиксированные положительные числа. Таким образом пространство параметров и $ $supp$\left(\beta(\delta_{2^{-k_{0}}}(\cdot)) \right)$ компактные множества. Пусть, $\sigma =\sigma ^{0} $, $\left|\sigma ^{0} \right|=c$ фиксированный вектор и  $\tau =\tau ^{0}$ фиксированная точка. Тогда $\Phi _{0} \left(\tau ,\sigma ,\rho \right)$- достаточно малая гладкая деформация следующей функции
\[\Phi =\sigma _{10}^{0} \tau _{1} +\sigma _{01}^{0} \tau _{2} +\sigma _{20}^{0} \tau _{1}^{2} +\tau _{1} \tau _{2}^{2} .\]
Если $\frac{\partial \Phi \left(\tau _{1}^{0} ,\tau _{2}^{0} \right)}{\partial \tau _{1} } \ne 0$ или $\frac{\partial \Phi \left(\tau _{1}^{0} ,\tau _{2}^{0} \right)}{\partial \tau _{2} } \ne 0$, то при $\left|\sigma -\sigma _{0} \right|<\varepsilon $  $\left|s_{30} \right|+\left|s_{21} \right|+\left|s_{12} \right|+\left|s_{03} \right|<\varepsilon $ справедлива следующая оценка: $\left|\nabla \Phi _{0} \left(\tau ,\sigma ,s\right)\right|>\delta >0$

\noindent  для некоторого положительного числа $\delta $.

        Применяя формулу интегрирования по частям для интеграла $J_{0}^{\chi } $, получим:
\begin{equation} \label{GrindEQ__3_}
\left|J_{0}^{\chi } \right|\le \frac{c\left\| a\right\| _{C^{1} } }{\left|\lambda \right|^{\frac{2}{3}}} ,
\end{equation}
где
\begin{equation} \label{GrindEQ__4_}
J_{0}^{\chi } \left(\lambda \right)=\int _{R^{2} }\chi \left(\tau \right)a\left(\tau _{1} ,\tau _{2} \right)\chi _{0} \left(\tau _{1} ,\tau _{2} \right)e^{i\lambda \rho \Phi _{0} \left(\tau ,\sigma ,s\right)} d\tau
\end{equation}
и $\chi -$гладкая функция сосредоточенная в достаточно малой окрестности  точки $\tau ^{0} .$  Достаточно рассмотреть случай когда $\tau ^{0}-$критическая точка.

 Так как $\tau ^{0} $- критическая точка, то справедливо следующее равенство:
\[\sigma _{10}^{0} +2\sigma _{20}^{0} \tau _{1}^{0} +2(\tau _{2}^{0})^{2} =0 ,  \sigma _{01}^{0} +2\tau _{1}^{0} \tau _{2}^{0} =0.\]
Для функции $\Phi $ в точке $\left(\tau _{1}^{0} ,\tau _{2}^{0} \right)$ матрица Гессе имеет вид:
\[\left(\begin{array}{l} {2\sigma _{20}^{0} } \\ {2\tau _{2}^{0} } \end{array}\right. \left. \begin{array}{l} {2\tau _{2}^{0} } \\ {2\tau _{1}^{0} } \end{array}\right).\]

Отметим, что это ненулевая матрица, так как если $\sigma _{20}^{0}=0$, то либо $\sigma _{10}^{0}\neq0$, либо $\sigma _{01}^{0}\neq0$, следовательно
$\tau _{1}^{0}\neq0$ или $\tau _{2}^{0}\neq0.$

Таким образом, ранг матрицы Гессе $\left(\begin{array}{l} {2\sigma _{20}^{0} } \\ {2\tau _{2}^{0} } \end{array}\right. $$\left. \begin{array}{l} {2\tau _{2}^{0} } \\ {2\tau _{1}^{0} } \end{array}\right)$ не меньше единицы. Если ранг матрицы равен единице, то, применяя лемму Морса по параметрам, для интеграла $J_{0} $ получим следующую оценку
\[\left|J_{0} \right|\le \frac{c\left\| a\right\| _{C^{1} }\rho^{\frac{1}{6}} }{\left|\lambda \right|^{\frac{1}{2} } } .\]

Наконец, суммируя полученные оценки придем к доказательству теоремы \ref{theorem1}.
Теорема \ref{theorem1} доказана.

В заключении авторы выражает свою глубокую благодарность рецензенту за ценные замечания.


\bigskip
\begin{thebibliography}{1}

%\bibitem{1} А.Ф. Леонтьев,
%{\it Ряды экспонент},  М.: Наука. 1983. 175 с.
%\bibitem{LUTS} В.И. Луценко, Р.С. Юлмухаметов,
%{\it Обобщение теоремы Пэли-Винера на весовые пространства}
%// Математ. заметки. Т.48, вып 5. 1990. С.~80-85.


\bibitem{Arn} Арнольд В.И., Варченко А.Н., Гусейн-заде С.М. {\it Особенности дифференцируемых отображений. Классификация критических точек, каустик и вольновых фронтов}, М.:Наука. 1982.
\bibitem{Varch}  Варченко А.Н. {\it Многогранник Ньютона и оценки осциллирующих  интегралов} //Функц.анал. и его прил.Т.10, вып 5. 1976. С.~13-38.
\bibitem{Vladim} Владимиров В.С. {\it Уравнения математической физики} М.:Наука. 1981.
\bibitem{VanDer}  Van der Korput. {\it K.G. Zur Methode der stationaren phase}// Compositio Math. V.1. 1934. P.~15-38.
\bibitem{Duistermat} Duistermaat J. {\it Oscillatory integrals Lagrange immersions and unifoldings of singularities} //
Comm. Pure.Appl.Math. - 1974. - V.27, № 2. - P.207-281.
\bibitem{Karpushkin1983} В.Н.Карпушкин. {\it Равномерные оценки осциллирующих интегралов с параболической или гиперболической фазой} // Труды Семинара имени И.Г.Петровского. вып.9. 1983. С.~3-39.
\bibitem{Sogge}  Sogge C.D., {\it Fourier integrals in Classical Analysis}, Cambridge, Cambridge university press, 1993. P.105.
\bibitem{Carberry} Carbery A.,  Christ M., and Wright J. {\it Multidimensional  Van der Korput lemma and sublevel set estimates}// Journal of AMS, V.12. 1999. P.981-1015.
\bibitem{Safarov} А.Сафаров {\it Инвариантные оценки двумерных осцилляторных интегралов} // Математические заметки. Т.104, вып 2. 2018. С.~289-300.

\end{thebibliography}
\end{document}